\documentclass[10pt]{amsart}
\usepackage{amscd,amsmath,amssymb,amsfonts}
\usepackage[all]{xy}
\usepackage{hyperref}
\usepackage{url}
\usepackage{stmaryrd}
\usepackage{color}

\theoremstyle{plain}
\newtheorem{thm}{Theorem}
\newtheorem{lem}[thm]{Lemma}

\newtheorem{prop}[thm]{Proposition}

\theoremstyle{definition}

\newtheorem{claim}[thm]{Claim}
\newtheorem{rmk}[thm]{Remark}

\numberwithin{thm}{section}

\newcommand{\ml}[2]{\begin{multline}\label{#1}#2 \end{multline}}
\newcommand{\ga}[2]{\begin{gather}\label{#1}#2 \end{gather}}

\newcommand{\sE}{{\mathcal E}}

\newcommand{\C}{{\mathbb C}}

\newcommand{\G}{{\mathbb G}}

\renewcommand{\P}{{\mathbb P}}
\newcommand{\Q}{{\mathbb Q}}

\newcommand{\Z}{{\mathbb Z}}

\begin{document}

\title[rigid systems]{ Rigid non-cohomologically rigid local systems   }
\author{Johan de Jong  \and H\'el\`ene Esnault \and Michael  Groechenig}
\address{ Mathematics Hall, 2990 Broadway, New York, NY 10027, USA }
 \email{ dejong@math.columbia.edu}
\address{Freie Universit\"at Berlin, Arnimallee 3, 14195, Berlin,  Germany}
\email{esnault@math.fu-berlin.de}
\address{  Mathematics,  
University of Toronto,  
Bahen Centre, Toronto, Ontario 
Canada 
 }
\email{michael.groechenig@utoronto.ca}

\thanks{ }

\begin{abstract} 
For any even natural number  $r \ge 2$, we construct an irreducible  rigid non-cohomologically  rigid complex local system  of rank $r$ on a smooth projective variety depending on $r$. For $r=2$, we construct an  irreducible rigid non-cohomogically rigid local system of rank $2$ on a quasi-projective variety   which becomes cohomologically rigid after fixing the conjugacy classes of the monodromies at infinity. \end{abstract}
\maketitle
\section{Introduction}
In \cite[Theorem~1.1]{EG18} it is  proven that on a smooth complex quasi-projective variety $X$, 
 irreducible cohomologically rigid local systems  with a fixed torsion determinant and fixed 
quasi-unipotent conjugacy classes of the monodromies at infinity  are integral. This answers positively a conjecture of Simpson  \cite[p.9]{Sim92} (formulated in the projective case)  under the cohomological rigidity condition, which simply means that the Zariski tangent space of the considered moduli space at the point corresponding to the local system is $0$. This property is used in a crucial way in the proof {\it loc. cit.} However, Simpson's conjecture concerns irreducible rigid local systems,  that is those for which  the corresponding moduli point is isolated, but perhaps fat.
In dimension $1$, Katz proved that all  irreducible rigid local systems  are cohomologically rigid  \cite[Cor.~1.2.5]{Kat96}. 
The irreducible local systems on a Shimura variety of real rank $\ge 2$ are also cohomologically rigid, see  \cite[Remark~A.2]{EG21}. 
In this short note we prove that
\begin{itemize}
\item[i)]  for any even natural number $r \ge 2$,  there exists
 a smooth complex projective variety depending on $r$  with an   irreducible  rigid non-cohomologically  rigid complex local system of rank $r$ on it,\\
\item[ii)]   and there exists a smooth complex quasi-projective variety  with
an irreducible rigid non-cohomogically rigid local system  of rank $2$ which becomes cohomologically rigid after fixing the conjugacy classes of the monodromies at infinity.
\end{itemize} 
The examples are an adaptation to algebraic geometry of \cite[(2.10.4)]{LM85} in which Lubotzky and Magid define the semi-direct product $\Gamma_2$ coming from the standard representation of the symmetric group in $3$ variables and show that the so defined rank $2$ representation $\Lambda_2$ of $\Gamma_2$  is rigid, see {\it loc.cit} (2.10.2), and has ${\rm dim} \ H^1(\Gamma_2, \sE nd(\Lambda_0))=1$.\\ \ \\
 In   our initial  construction, the monodromy was finite, in particular the Zariski closure of the monodromy group  is equal to  itself,  the determinant is finite, and the monodromies at infinity are finite as well. Alexander Petrov noticed that we could  apply K\"unneth formula to the exterior product of our examples with a standard cohomologically rigid local system with infinite monodromy to obtain examples as in (i) and (ii)  (with
 different ranks) with infinite monodromy, see Remark~\ref{rmk:sasha}.
 \\ \ \\
 {\it Acknowledgements}: We thank Vasily Rogov for bringing the example of Lubotszky and Magid {\it loc. cit.} to our attention. Furthermore, it is a pleasure to acknowledge helpful comments by Piotr Achinger, Adrian Langer, Aaron Landesman, and Daniel Litt.   We  are grateful to Alexander Petrov for his interest and for contributing to this note in the form of Remark~\ref{rmk:sasha}.

\section{The standard representation of the symmetric group} \label{sec:sr}

We fix once and for all an even natural number $r\ge 2$. 
For $S_{r+1}$  the symmetric group on $(r+1)$ elements,  we define the standard representation
\ga{}{ \Lambda_r={\rm Ker} (\Z^{r+1}\xrightarrow{\Sigma}  \Z), \  \underline{a}=(a_1, \ldots,a_{r+1})\mapsto \sum_1^{r+1} a_i=\Sigma(\underline{a})  \notag  \\
\rho_r: S_{r+1}\to GL(\Lambda_r)  \notag}
by permuting the elements. This  rank $r$ representation  has among others  two properties:
\begin{itemize}
\item[$(\star)$ ]  $\Lambda_r$ is absolutely irreducible, i.e. $\Lambda_e\otimes_{\Z}\C$ is irreducible;
\item[$(\star \star)$] $\Lambda_r\otimes_{\Z}\C$ is a direct summand of ${\sE nd}^0(\Lambda_r \otimes_{\Z}\C)$.

\end{itemize}
The property $(\star) $ is very classical and is straightforward to check:  the basis vectors  $(e_i-e_{r+1})_{i=1,\ldots, r}$ of $\Lambda_0$ are permuted by the permutations $(i,r+1)_{ i=1,\ldots, r},$where 
 $(e_i)_{ i=1,\ldots, r+1}$ is the standard basis of  $\Z^{r+1}$. The property $(\star \star$) is verified for $r=2$  by the character table in
 \cite[(2.10.4)]{LM85} and in \cite[Ex.1.2]{BDO15} in general.  Here ${\sE nd}^0$ denotes the trace-free endomorphisms. 
 This defines the semi-direct product 
\ga{}{ \Gamma_r:= \Lambda_r \rtimes S_{r+1} \notag \\
0\to \Lambda_r\to \Gamma_r\to S_{r+1}\to 1.\notag }

\section{A rigid non-cohomologically rigid complex system on a smooth projective variety} \label{sec:proj}

Let $E$ be an elliptic curve. We define the abelian variety
 \ga{}{ A_r= E\otimes_{\Z} \Z^{r}= \underbrace{E\times \ldots \times E}_{r-{\rm times}}  \notag}
 of dimension $r$. 
 Via the basis $(\epsilon_i=e_i-e_{r+1})_{i=1,\ldots, r}$ of Section~\ref{sec:sr}, 
 we define $S_{r+1}$ as the subgroup of ${\rm Aut}(A_r)$ defined by
\ga{}{ S_{r+1}=\mathbb{I}_{E}\otimes S_{r+1}\xrightarrow{\mathbb{I}\otimes \rho_r 
 \  {\rm injective}} {\rm Aut}(E) \otimes_{\Z} GL(\Lambda_r)\subset {\rm Aut}(A_r). \notag}
For example, setting $\sigma=(12 \ldots (r+1))$
\ml{1}{ \sigma \big(x_1,\ldots, x_r\big)=\big(-(x_1+x_2+\ldots+x_r), x_1,x_2, \ldots, x_{r-1}\big)\\
{\rm for} \ (x_1,\ldots, x_r)\in E\otimes_{\Z} \Z^r. }
In particular, the origin $0_A$ of the abelian variety is a fixpoint  for this action.  

We now do Serre's construction as in \cite[Sect.~15]{Ser58}.
Let $P_r$ be a smooth projective simply connected variety over $\C$ on which $S_{r+1}$ acts without fixpoint (\cite[Prop.~15]{Ser58}).  We define \ga{}{  Y_r= A_r\times_{\C} P_r, \ X_r=Y_r/S_{r+1} \notag}  where $S_{r+1}$ acts diagonally. This yields the $S_{r+1}$-Galois cover 
\ga{}{\pi_r: Y_r \to X_r\notag}
and the associated Galois exact sequence 
\ga{}{0\to \pi_1(Y_r,y_r)=H_1(A_r)=H_1(E)\otimes_{\Z} \Lambda_0 \to \pi_1(X_r,x_r) \to S_{r+1}\to 1\notag}
where $y_r=(0_A,p_r)$ is a  $\C$-point of $Y_r$ and $x_r=\pi_r(y_r)$.  The Galois exact sequence is the same as the homotopy exact sequence of the smooth projective morphism
\ga{}{  f_r: X_r\to P_r/S_{r+1} \notag}
 with fibres isomorphic to $A$ and with section
 \ga{}{ P_r/S_{r+1}\xrightarrow{\cong} (0_A\times_{\C} P_r)/S_{r+1}\subset X.\notag}
 Thus 
 \ga{}{ \pi_1(X_r,x_r)= \pi_1(Y_r,y_r) \rtimes S_{r+1}\notag}
where  the action of $S_{r+1}$ on $\pi_1(Y_r,y_r)=H_1(E) \otimes_{\Z} \Z^r  $  is equal t to $ \mathbb{I}_{H_1(E) }\otimes_{\Z} \rho_r$.

 We denote by $\mathbb L_r$ the local system defined by the composite 
\ga{}{ \pi_1(X_r,x_r)\to S_{r+1}\xrightarrow{\rho_r} GL(\Lambda_r). \notag} 
We denote by  $M_{B}(X_r,2, {\rm det} ( \mathbb L_r))$ the Betti moduli space of rank $r$ irreducible local systems on $X_r$ with fixed determinant ${\rm det} (\mathbb L_r)$, by $M_{B}(X_r,r)$ the Betti moduli space of rank $r$ irreducible local systems on $X_r$ (\cite[Prop.~2.1]{EG18}) and by 
\ga{}{ \phi_r: M_B(X_r,r, {\rm det} (\mathbb L_r))\to M_B(X_r,r)\notag}
the closed immersion induced by omitting the restriction on the determinant.  We have
\ga{}{ \mathbb L_r\in M_B(X_r,r, {\rm det}(\mathbb L_r)).\notag}
\begin{prop} \label{prop:rig} When $r$ is even, 
$\mathbb L_r$ is rigid in $M_B(X_r,r, {\rm det}(\mathbb L_r))$ and $\phi_r(\mathbb L_r)$ is rigid in $M_B(X_r,r)$.
\end{prop}
\begin{proof}  It suffices to prove this for $\phi_r(\mathbb L_r) \in M_B(X_r,r)$ since $\phi$ is a closed immersion. Let  $T$ be an irreducible complex affine curve, together with a morphism
$\tau: T\to M_B(X_r,r)$ and a complex point $t_0\in T$ such that $\tau(t_0)=\mathbb L_r$. 
We fix a complex point $t\in T$, and denote by $\mathbb{L}_r(t)$ the irreducible complex local system corresponding to $\tau(t)$.   As $H_1(A_r)$ is normal in 
$\pi_1(X,x)$, Clifford theory implies that 
 $\mathbb{L}_r(t)|_{H_1(A_r)}$  is a sum of $r$ characters $\chi_{t,1} \oplus  \ldots \oplus \chi_{t,r}$, with $\chi_{t,i}\in 
 {\rm Hom}(H_1(A_r), \C^*)$.  As $\mathbb{L}_r(t)|_{H_1(A_r)}$  is $S_{r+1}$-invariant, the subset
 $\{\chi_{t,1}, \ldots, \chi_{t,r} \}\in {\rm Hom}(H_1(A_r), \C^*)$ consists of a union of $S_{r+1}$-orbits. 

\begin{claim} \label{claim:tor} If $r$ is even, the characters
$\chi_{t,i}$ for $i=1,\ldots, r$ are torsion characters, that is they lie in  
\ga{}{ {\rm Hom}(H_1(E)\otimes_{\Z} \Z^r, 
2\pi i \Q/2 \pi i \Z) \subset  {\rm Hom}(H_1(A_r), \C^*).\notag}

\end{claim}
\begin{proof}
Each single orbit of $S_{r+1}$ in  the set $\{\chi_{t,1}, \ldots, \chi_{t,r}\}$
has length $s\le r < r+1$ and defines a quotient $S_{r+1}\to  S_s$, where $s=1$ or $ 2$ if $r=2$. For $r\ge 4$,  the only non-trivial normal subgroup of $S_{r+1}$ is the alternate group.  
 Thus the image of $S_{r+1}$ in $S_s$ has order $1$ or $2$ in all cases. This implies  that  $\sigma$, which   has odd order, must map to the identity in $S_s$, that is 
\ga{2}{ \sigma (\chi_{t,i})=\chi_{t,i} \ {\rm for} \ i=1,\ldots, r.}
Pick $0 \neq \gamma \in H_1(E)$ and some $i \in \{1,\ldots, r\}$.  Write for $(x_1,\ldots x_r)\in \Z^r$
\ga{}{ \chi_{t,i}(\gamma\otimes x_1,\ldots \gamma\otimes x_r)={\rm exp}(2\pi i (a^1x_1 +\ldots + a^r x_r))\notag}
for some $a^1,\ldots, a^r\in \C$.
Then  \eqref{1} and \eqref{2} yield
\ga{}{ {\rm exp}(2\pi i (a^1x_1 +\ldots + a^r x_r))
= {\rm exp}( 2\pi i(-a^1(x_1+\ldots +x_r)  + a^2 x_1 +\ldots + a^r x_{r-1})).
\notag}
As this is true for all $x_j \in \Z$ we derive
\ga{}{ a^j\in \frac{1}{(r+1)}\Z \ {\rm for \ all} \ j=1,\ldots, r,\notag}
thus  $\chi_{t,i}$ has order dividing $(r+1)$. 
This  finishes the proof.

\end{proof}
The scheme map $T\to  {\rm Hom}(H_1(A_r), \C^*)/S_{r+1}$ which sends $t$ to  $(\chi_{t,1}, \ldots, \chi_{t,r})$ up to ordering, has by Claim~\ref{claim:tor}  values in $ ( {\rm Hom}(H_1(A_r), 2\pi i \Q/2 \pi i \Z))^r/S_r$.  This is a countable subset of complex  points, thus has to be finite. As $T$ is connected, it is constant, thus equal to its value at $t=t_0$, thus is the image of  the trivial character of rank $r$. Thus $\mathbb{L}_r(t)$ factors through $S_{r+1}$. Since there are only finitely many complex points in $M_B(X,2)$ the monodromy representation of which factors through $S_{r+1}$, $\mathbb{L}_r$ is rigid. This concludes the proof for $\phi(\mathbb{L}_r)$. For $\mathbb{L}_r$, we assume that $\tau$ factors through $\tau_0: T\to M_B(X,r,{\rm det} (\mathbb{L}_r))$ and we argue with $\mathbb{L}_r(t) \in M_B(X,r, {\rm det}(\mathbb{L}_r))$ corresponding to $\tau_0(t)$ in place of $\mathbb{L}(t)$, which does not change the argument. This concludes the proof.
\end{proof}
\begin{prop} \label{prop:mainex}  Under the assumption of Proposition~\ref{prop:rig}, 
$\mathbb L_r$ is not cohomologically  rigid in $M_B(X_r,r, {\rm det}(\mathbb{L}_r))$, and $\phi(\mathbb L_r)$ 
 is not cohomologically  rigid in $M_B(X_r,r)$.

\end{prop}
\begin{proof}
The Hochschild-Serre spectral sequence for $\pi_1(X_r,x_r)$ reads
\ml{}{ H^1(X_r, \sE nd(^0\mathbb{L}_r))= H^1(H_1(A_r), \sE nd^0(\mathbb{L}_r))^{S_{r+1}}
= \\ {\rm Hom}_{S_{r+1}}(H_1(A_r), 
\sE nd^0(\mathbb{L}_r))= {\rm Hom}_{S_{r+1}}( \underbrace{\Lambda_r\oplus \ldots \oplus \Lambda_r}_{r-{\rm times}}, 
\sE nd^0(\mathbb{L}_r)) \neq 0. \notag}
The non-vanishing on the right comes from Property $(\star \star)$. The Zariski tangent space to $M_B(X_r,r, {\rm det}(\mathbb L_r))$ at $\mathbb L_r$ is precisely $H^1(X_r, \sE nd^0(\mathbb{L}_r))$, which is a direct summand
$H^1(X_r, \sE nd(\mathbb{L}_r))$, which in turn is the Zariski tangent space to $M_B(X_r,r)$ at  $\phi(\mathbb L_r)$. 
This concludes the proof. 
\end{proof}
\begin{rmk}
In Proposition~\ref{prop:mainex},  $\mathbb L_r$ by definition is of the shape $f_r^* \mathbb L'_r$, for an irreducible
local system $\mathbb L'_r$ on $P_r/S_r$, which is uniquely defined.  But 
\ga{}{ H^1(P_r/S_r, \sE nd^0(\mathbb L'_r))=
H^1(S_r, \sE nd^0(\mathbb L'_r))
=0 \notag} 
as $S_r$ is finite, thus $\mathbb L'_r$ is cohomologically rigid on $P_r/S_r$.
Given an irreducible local system $\mathbb L$ defined by the representation $\rho: \pi_1(X,x)\to GL_r(\C)$ on $X$ smooth projective,
we could dream of the existence of a factorization 
  $\rho: \pi_1(X,x)\to \pi_1(Y,y)\xrightarrow{\rho'}  GL_r(\C)$ where $Y$ is smooth projective, defining $\mathbb L'$ on $Y$, with the property that $\mathbb L'$ is cohomologically rigid. Even without requesting $\pi_1(X,x)\to \pi_1(Y,y)$ to come from geometry, we have no access to it. Together with \cite[Theorem~1.1]{EG18}  it would prove Simpson's integrality conjecture in general  on $X$ smooth projective. 
\end{rmk}
\section{ The quasi-projective case } \label{sec:qproj}
We now assume $r=2$. 
We replace $A_2$ in the previous section with the $2$-dimensional torus 
\ga{}{ T=\G_m^{3}/({\rm diagonal} \  \G_m) \notag}
  which is the Jacobian of $\P^1\setminus \{0,1,\infty\}$. We take the same $P_2$ as in Section~\ref{sec:proj} and define $V=T\times_\C P_2, \ U=V/S_3$ with the $S_3$-Galois cover $q: V\to U$. This yields the Galois exact sequence
\ga{}{0\to \pi_1(V,v)=\pi_1(T)=\Z^2\to \pi_1(U,u)\to S_3\to 1\notag}
based at a point $(1,1,1)\times p=v$ with $u =q(v)$, 
which is identified with  the homotopy exact sequence of $V\to P_2/S_3$ with fibre $T$.  This fibration has a section $\{1,1,1\}\times P/S_3,$ thus
\ga{}{ \pi_1(U,x)= \Gamma_2.\notag}
Thus $\rho_2$ defines a local system $\mathbb{M}$ on $U$.  Again we introduce the moduli spaces 
$M_B(U,2, {\rm det}( \mathbb M))$ and $M_B(U,2)$ with the forgetful map  
\ga{}{ \varphi: M_B(U,2, {\rm det} ( \mathbb M)) \to M_B(U,2).\notag}
\begin{lem}  \label{lem:open}
$\mathbb{M}$ is rigid and not cohomologically rigid in $M_B(U,2, {\rm det}(\mathbb M))$, $\varphi(\mathbb M)$  is rigid and not cohomologically rigid in $M_B(U,2)$.
\end{lem}
\begin{proof} 
For $\varphi(\mathbb M)$ the rigidity is proved in 
\cite[Lemma~2.11]{LM85} and is in fact a consequence of $M_B(U,2)$ being $0$-dimensional (\cite[(2.10.1)]{LM85}), which also shows that $\mathbb M$ is rigid. 
The cohomological statement $H^1(\Gamma_2, \sE nd^0 (\mathbb{M}))=\C$ is proved in 
  \cite[2.10.4]{LM85}, which is a special case of Property $(\star \star)$.  Hence, $\mathbb M$ and $\varphi(\mathbb M)$ are not cohomologically rigid.   
    \end{proof}
 
We now fix  the monodromies at infinity to be those of $\mathbb M$. As the global monodromy of $\mathbb M$ is finite, so are the monodromies at $\infty$, in particular they are quasi-unipotent.  We define
 $M_B(U,2,  {\rm det} (\mathbb M) , {\rm mon}(\mathbb M))$ as in~\cite[Section~2]{EG18}. 
 By  \cite[Remark~2.4]{EG18}, the Zariski tangent space at $\mathbb M$ of $M_B(U,2, {\rm det} (\mathbb M), {\rm mon}(\mathbb M))$ is equal to 
 \ga{3}{ {\rm Ker} \big( H^1(U, \sE nd^0(\mathbb M)) \xrightarrow{\rm restriction}  \oplus_{\gamma} H^1( \langle \gamma \rangle, \sE nd^0(\mathbb M))\big),}
 where the $\gamma$  are small loops around the components at $\infty$ of a smooth projective 
  compactification $U\hookrightarrow \bar U$ with strict normal crossings at infinity, and $\langle \gamma \rangle$ is the free commutative group spanned by it. We denote by $\mathbb M_0$ the local system $\mathbb M$ viewed in  $M_B(U,2,  {\rm det} (\mathbb M), {\rm mon}(\mathbb M))$.

\begin{lem} \label{lem:boundary}
$\mathbb{M}_0$ is cohomologically rigid.
\end{lem}
\begin{proof} 
By \eqref{3},  it is enough to find one loop $\gamma$ at infinity such that the restriction map
\ga{}{ H^1(\pi_1(U,x), \sE nd^0(\mathbb{M})) \to H^1(\langle \gamma \rangle, \sE nd^0(\mathbb{M}))\notag}
is injective.  The action of $S_3$ on 
on $\P^2\setminus T$, which is a union of the $3$ coordinate lines, permutes them. This defines a $S_3$-equivariant compactification 
\ga{}{ V=T\times_{\C} P\hookrightarrow  \bar V=\P^2\times_\C P,\notag}
 thus a  compactification 
 \ga{}{ j: U\hookrightarrow \bar U=\bar V/S_3 \notag}
  where $S_3$ acts on $\bar V$ diagonally. This yields the commutative diagram
\ga{}{  \xymatrix{ P_2/S_3 \ar[dr]^=  \ar[r] &\ar[d] U\ar[r] & \ar[d] \bar U\\
& P_2/S_3 \ar[r]^= & P/S_3 }\notag}
with fibres $T\hookrightarrow \P^2$ above  the image $\bar p$  of $p\in P_2 $ in $P_2/S_3$. We take $\gamma$ in the fibre $\P^2$ which winds around one of the irreducible components of 
$\P^2\setminus T$. Thus, it maps to $1$ in $\pi_1(P_2/S_3,\bar p)=S_3$, that is  $\gamma\in \Z^2=\pi_1(T)\subset \pi_1(U,x)$.  Thus $\mathbb M$ restricted  to $\langle \gamma \rangle$ is trivial. 
The restriction map 
\ga{}{ H^1(\pi_1(U,u), \sE nd^0(\mathbb{M}))
\xrightarrow{\rm rest}  H^1(\langle \gamma \rangle, \sE nd(^0\mathbb{M})) 
\notag}
is then identified with
\ga{}{{\rm Hom}_{S_3}( \mathbb{M}, \sE nd^0(\mathbb{M}))  \to {\rm Hom}( \Z, \sE nd^0(\mathbb{M})) \notag}
via the inclusion   $\Z\langle \gamma \rangle \to \Z^2$ of $\Z$-modules.  Therefore,  the map $\rm rest$ is injective.  This concludes the proof.
\end{proof}

\begin{rmk} \label{rmk:sasha}
This remark is due to Alexander Petrov. Let $S$ be a smooth complex projective variety, $\mathbb L_S$  be an irreducible rigid local system on $S$. We denote by $\mathbb L'_r$ 
the exterior product $\mathbb L_r\boxtimes \mathbb L_S$ on $X_r\times_{\C} S$  in Proposition~\ref{prop:rig}, by  $\mathbb L'$,  resp. $\mathbb L'_0$ 
the exterior product $\mathbb L \boxtimes \mathbb L_S$ on $U\times_{\C} S$  viewed in 
$M_B(U\times_{\C} S,2\cdot {\rm rank}(\mathbb L_S), {\rm det}(\mathbb \mathbb L'))$, resp.
$M_B(U\times_{\C} S,2\cdot {\rm rank}(\mathbb L_S), {\rm det}(\mathbb \mathbb L')), {\rm mon}(\mathbb L'))$ 
in  Lemma~\ref{lem:open}, resp.  Lemma~\ref{lem:boundary}. Then 
\begin{claim} \label{claim:sasha}
 Proposition~\ref{prop:mainex} remains true with $\mathbb L_r$ replaced with $\mathbb  L'_r$ and  Lemma~\ref{lem:open} and Lemma~\ref{lem:boundary} remain true with $\mathbb L$ replaced with $\mathbb L'$.
\end{claim}
\begin{proof}
We applied K\"unneth formula to $H^1$ of $\sE nd(\mathbb L'_r)= \sE nd(\mathbb L_r) \boxtimes \sE nd(\mathbb L_S)$. This yields the cohomological part of the statement. Similarly for $\mathbb L'$. As for rigidity:  In a $T$-deformation 
  $(\mathbb L'_r)_t$  of $\mathbb L'_r$ as in Proposition~\ref{prop:rig}, the restriction of 
  $(\mathbb L'_r)_t$ to $x\times S$ for any complex point $x\in X_r$ is isomorphic to the restriction of 
    $(\mathbb L'_r)_{t_0}$ to $x\times S$  by rigidity of 
  $\mathbb L_S$, and similarly  the restriction of 
  $(\mathbb L')_t$ to $X_r\times s$ for any complex point $s\in S$  is isomorphic to the restriction of 
  $(\mathbb L')_{t_0}$ to $X_r\times s$  by rigidity of $ \mathbb L_r$. The K\"unneth property for  $\pi_1(X_r\times_{\C} S)$ implies that  $(\mathbb L'_r)_t \cong  (\mathbb L'_r)_{t_0}  $. Similary for $\mathbb L'$. 

\end{proof}
In particular, we can take $S$ to be a positive projective Shimura variety of real rank $\ge 2$, on which all the irreducible local systems are rigid, even cohomologically rigid, and take $\mathbb L_S$ to be a rank $\ge 2$ irreducible local system.
In this way,  we  obtain examples as in (i) and (ii)  (with
 different ranks) with infinite monodromy.

\end{rmk}

\end{document}